\documentclass[12pt, a4paper, twosided]{article}
\usepackage{bbm}
\usepackage{pifont}
\usepackage{bbding}
\usepackage{mathrsfs}
\usepackage{pgf,tikz}
\usepackage{amsfonts,amssymb,amsmath,indentfirst,amsthm}
\usepackage{epsfig}
\usepackage{fancyhdr}
\usepackage{graphicx}
\usepackage{cite}
\usepackage{url}
\usepackage{xcolor}
\usepackage{cite}
\usepackage{calc}
\usepackage{cases}
\usetikzlibrary{patterns}

\def\ess~inf{\mathop{\rm ess~inf}}

\numberwithin{equation}{section}

\newenvironment{key words}{\emph{\texttt{Keywords}}\mbox{  }}{ }

\newtheorem*{theorema}{Theorem A}
\newtheorem*{theoremb}{Theorem B}
\newtheorem*{theoremc}{Theorem C}
\newtheorem{theorem}{Theorem}[section]
\newtheorem{lemma}[theorem]{Lemma}

\renewenvironment{proof}{\noindent{\textbf{Proof.}}}{\hfill$\Box$}
\theoremstyle{remark}

\theoremstyle{plain}

\makeatletter

\newcommand{\Rmnum}[1]{\expandafter\@slowromancap\romannumeral #1@}
\makeatother

\usepackage{comment}

\allowdisplaybreaks

\begin{document}

\fancyhf{}

\fancyhead[EC]{X. Zhu, W. Li}

\fancyhead[EL]{\thepage}

\fancyhead[OC]{Endpoint regularity of general Fourier integral operators}

\fancyhead[OR]{\thepage}

\renewcommand{\headrulewidth}{0.5pt}
\renewcommand{\thefootnote}{\fnsymbol {footnote}}
\title
{\textbf{Endpoint regularity of general Fourier integral operators} \thanks{This work was supported by the National Key Research and Development Program of China, grant number 2022YFA1005700
		and the National Natural Science Foundation of China (No. 12271435, No. 11871436).}}

\author{Xiangrong Zhu \\
	\small{School of Mathematical Sciences, Zhejiang Normal University,  Jinhua 321004, P.R. China}\\
	Wenjuan Li \thanks{Corresponding author's email address: liwj@nwpu.edu.cn }\\
	\small{School of Mathematics and Statistics, Northwestern Polytechnical University,
		Xi'an, 710129, China} }

\date{}
 \maketitle

 {\bf Abstract:}  Let $n\geq 1,0<\rho<1, \max\{\rho,1-\rho\}\leq \delta\leq 1$ and
 $$m_1=\rho-n+(n-1)\min\{\frac 12,\rho\}+\frac {1-\delta}{2}.$$
 If the amplitude $a$ belongs to the H\"{o}rmander class $S^{m_1}_{\rho,\delta}$
 and $\phi\in \Phi^{2}$ satisfies the strong non-degeneracy condition, then we prove that the following  Fourier integral operator $T_{\phi,a}$ defined by
 \begin{align*}
 	T_{\phi,a}f(x)=\int_{\mathbb{R}^{n}}e^{i\phi(x,\xi)}a(x,\xi)\widehat{f}(\xi)d\xi,
 \end{align*}
 is bounded from the local Hardy space $h^1(\mathbb{R}^n)$ to $L^1(\mathbb{R}^n)$. As a corollary, we can also obtain the corresponding $L^p(\mathbb{R}^n)$-boundedness when $1<p<2$.
 
 These theorems are rigorous improvements on the recent works of Staubach and his collaborators. When $0\leq \rho\leq 1,\delta\leq \max\{\rho,1-\rho\}$, by using some
 similar techniques in this note, we can get the corresponding theorems which coincide with the known results.



\section{Introduction and main results}

 A pseudo-differential operator (PDO for short) is given by
$$T_a f(x)=(2\pi)^{-n}\int_{\mathbb{R}^n}e^{ix\cdot\xi}a(x,\xi)\widehat{f}(\xi)d\xi,$$
where $\widehat{f}$ is the Fourier transform of $f$ and $a$ is the amplitude. We always omit the constant $(2\pi)^{-n}$ throughout this note.

In its basic form, a Fourier integral operator (FIO for short) is defined by
\begin{align*}
	T_{\phi,a}f(x)=\int_{\mathbb{R}^{n}}e^{i\phi(x,\xi)}a(x,\xi)\widehat{f}(\xi)d\xi,
\end{align*}
where $\phi$ is the phase. In this note, we always assume that $f$ belongs to the Schwartz class $S(\mathbb{R}^{n})$.

A FIO $T_{\phi,a}$ is simply a pseudo-differential operator if $\phi(x,\xi)=x\cdot \xi$. When $\phi(x,\xi)=x\cdot \xi+|\xi|$, $T_{\phi,a}$ is closely related to the wave equation and Fourier transform on the unit sphere in $\mathbb{R}^{n}$ (see \cite[p. 395]{S93}).

FIOs have been widely used in the theory of partial differential equations and micro-local analysis. For instance, the solution to an initial value problem for a hyperbolic equation with variable coefficients can be effectively approximated by an FIO of the initial value (see \cite[p. 425]{S93}). Therefore, the boundedness of related FIOs provides a priori estimate for the solution. A systematic study of
these operators was initiated by H\"{o}rmander \cite{H71AM}.

At first, we recall some simplest and most useful definitions on amplitudes and phases.

Let $\mathbb{N}$ be the set $\{0,1,2,\ldots\}$. A function $a$ belongs to the H\"{o}rmander class $S^{m}_{\rho,\delta}$ $(m\in \mathbb{R},0\leq\rho,\delta\leq1)$ if it satisfies
\begin{equation}\label{gsh1.1}
	\sup_{x,\xi\in\mathbb{R}^{n}}(1+|\xi|)^{-m+\rho N-\delta M}|\nabla^{N}_{\xi}\nabla^{M}_{x}a(x,\xi)|=A_{N,M}<+\infty
\end{equation}
for any $N,M\in \mathbb{N}$. Immediately, one have
$$S^{m_1}_{\rho,\delta}\subset S^{m_2}_{\rho,\delta},S^{m}_{\rho_2,\delta}\subset S^{m}_{\rho_1,\delta},S^{m}_{\rho,\delta_1}\subset S^{m}_{\rho,\delta_2},$$
if $m_1<m_2,\rho_1<\rho_2,\delta_1<\delta_2.$

A real-valued function $\phi$ belongs to the class $\Phi^{2}$ if $\phi$ is positively homogeneous of order 1 in the frequency variable $\xi$ and satisfies
\begin{equation}\label{gsh1.2}
	\sup_{(x,\xi)\in\mathbb{R}^{n}\times(\mathbb{R}^{n}\setminus\{0\})}|\xi|^{-1+N}|\nabla^{N}_{\xi}\nabla^{M}_{x}\phi(x,\xi)|=B_{N,M}<+\infty
\end{equation}
for all $N,M\in \mathbb{N}$ with $N+M\geq 2$.

A real-valued function $\phi\in C^{2}(\mathbb{R}^{n}\times(\mathbb{R}^{n}\setminus\{0\}))$ satisfies the strong non-degeneracy condition (SND for short), if there exists a constant $\lambda>0$ such that
\begin{equation}\label{gsh1.3}
	\textrm{det}\left(\frac{\partial^{2}\phi}{\partial x_{j}\partial \xi_{k}}(x,\xi)\right)\geq \lambda
\end{equation}
for all $(x,\xi)\in\mathbb{R}^{n}\times(\mathbb{R}^{n}\setminus\{0\})$.

For PDOs and FIOs, the most important problem is whether they are bounded on Lebesgue spaces and Hardy spaces.
This problem has been extensively studied and there are numerous results. Here we always assume that $a\in S^m_{\rho,\delta}$ and $\phi\in \Phi^{2}$ satisfies the SND condition (\ref{gsh1.3}).

For PDOs, when $a\in S^{m}_{\rho,\delta}$, then $T_a$ is bounded on $L^2$ if either $m\leq\frac{n}{2}\min\{0,\rho-\delta\}$ when $\delta<1$
(see H\"{o}rmander \cite{H71CPAM}, Hounie \cite{H86} and Calder\'{o}n-Vaillancourt \cite{CV71,CV72}) or $m<\frac{n}{2}(\rho-1)$ when $\delta=1$ (see Rodino \cite{R76} ).
The bound on $m$ is sharp. For endpoint estimates, one can see Fefferman \cite{F73}, \'{A}lvarez-Hounie \cite{AH90} and Guo-Zhu \cite{GZ22}.

For the local $L^{2}$ boundedness of FIO, it can be date back to Eskin \cite{E70} and H{\"o}rmander \cite{H71AM}. The transference of local to global regularity of FIOs can be found in
Ruzhansky-Sugimoto \cite{RS19}. Among numerous results on the global $L^{2}$ boundedness, we would like to mention that Dos Santos Ferreira-Staubach \cite{FS14} proved the global $L^{2}$ boundedness if either $m\leq\frac{n}{2}\min\{0,\rho-\delta\}$ when $\delta<1$ or $m<\frac{n}{2}(\rho-1)$ when $\delta=1$. This bound on $m$ is also sharp. For more results, see for instance \cite{Af78,B97,F78,K76,RS06}.

For endpoint estimates of FIOs, Seeger-Sogge-Stein \cite{SSS91} proved the local $H^{1}$-$L^{1}$ boundedness for $a\in S^{(1-n)/2}_{1,0}$ and got the $L^{p}$ boundedness for
$a\in S^{m}_{1,0}$ when $m=(1-n)|\frac{1}{p}-\frac{1}{2}|$ by the Fefferman-Stein interpolation. Tao \cite{T04} showed the weak type (1,1) for $a\in S^{(1-n)/2}_{1,0}$.
For the regularity of FIOs and its applications, there has been a great deal of progress and work recently, for example Cordero-Nicola-Rodino \cite{CNR09,CNR10}, Coriasco-Ruzhansky \cite{CR10, CR14},
Hassell-Portal-Rozendaal \cite{HPR20}, Israelsson-Rodr\'{i}guez L\'{o}pez-Staubach \cite{IRS21} etc. Among these results, the latest result is the following theorem which was proved by Castro-Israelsson-Staubach \cite{CIS21}.
\begin{theorema}
	(\cite[Theorem 1]{CIS21}) Let $n\geq 1,0\leq \rho\leq 1, 0\leq \delta<1$ and $a\in S^{m}_{\rho,\delta}$ with
	$$m\leq (\rho-n)|\frac 12-\frac 1p|+\frac{n}{2}\min\{0,\rho-\delta\}.$$
	If $\phi\in \Phi^{2}$ satisfies the SND condition (\ref{gsh1.3}), then the FIO $T_{\phi,a}$ is bounded on $L^p$ for $1<p<\infty$.
\end{theorema}
When $2<p<\infty$, in \cite{SZ23} Shen-Zhu improved Theorem A to
$$m\leq (\rho-n)|\frac 12-\frac 1p|+\frac{n}{p}\min\{0,\rho-\delta\}.$$

Before the next theorem, we recall the definitions of Hardy spaces and local Hardy spaces.

Let $\Phi$ be a function in the Schwartz space $S(\mathbb{R}^n)$ satisfying $\int_{\mathbb{R}^n}\Phi(x)dx=1$.
Set $\Phi_t(x)=\frac{1}{t^n}\Phi(\frac{x}{t})$. Following Stein \cite[p. 91]{S93}, we can define the Hardy space $H^p(\mathbb{R}^n)(0<p<+\infty)$ as the space of all tempered distributions $f$ satisfying
$$\|f\|_{H^p}=\|\sup\limits_{t>0}|f\ast\Phi_t|\|_{L^p(\mathbb{R}^n)}<\infty.$$
The local Hardy space (see \cite {G79}) $h^p(\mathbb{R}^n)(0<p<+\infty)$ is defined as the space of all tempered distributions $f$ satisfying
$$\|f\|_{h^p}=\|\sup\limits_{0<t<1}|f\ast\Phi_t|\|_{L^p(\mathbb{R}^n)}<\infty.$$
It is well-known that $H^p=h^p=L^{p}$ for equivalent norms when $1<p<+\infty$ and $H^1\subset h^1\subset L^1$.

In \cite{IMS23}, Israelsson-Mattsson-Staubach extend Theorem A to $p\leq 1$.
\begin{theoremb}
	(\cite{IMS23}) Let $n\geq 1,0\leq \rho\leq 1, 0\leq \delta<1$ and $a\in S^{m}_{\rho,\delta}$ with
	$$m\leq (\rho-n)|\frac 12-\frac 1p|+\frac{n}{2}\min\{0,\rho-\delta\}.$$
	If $\phi\in \Phi^{2}$ satisfies the SND condition (\ref{gsh1.3}), then $T_{\phi,a}$ is bounded from the local Hardy space $h^p$ to $L^p$ when  $\frac {n}{n+1}<p\leq 1$.
	Furthermore, if $a$ has compact support in the spatial variable $x$, then $T_{\phi,a}$ is bounded from $h^p$ to $L^p$ for any $0<p\leq 1$.
\end{theoremb}
Theorem B can be considered as a corollary of \cite[Proposition 5.7]{IRS21}, \cite[Proposition 6.4]{IRS21} and \cite[Proposition 5.1]{IMS23}.

As one can see in the proof, the part $(\rho-n)|\frac 12-\frac 1p|$ comes from the $h^p-L^2$ boundedness of the fractional integral operator and the part $\frac{n}{2}\min\{0,\rho-\delta\}$ comes from the $L^2$ boundedness of the Fourier integral operator. It is natural to expect that Theorem B is sharp when $p<2$. However, recently, Ye-Zhang-Zhu \cite{YZZ23} extended Theorem B to $\delta=1$, and  got an unexpected result which yields that Theorem B is not sharp at least when $\delta$ is close to 1.

\begin{theoremc}
	(\cite[Theorem 1.1]{YZZ23})
	Suppose that $0\leq\rho\leq 1,n\geq 2$ or $0\leq\rho<1,n=1$, $a\in S^{m}_{\rho,1}$ with
	$$m\leq \frac{\rho-n}{p}+(n-1)\min\{\rho,\frac 12\}$$
	and $\phi\in \Phi^{2}$ satisfies the SND condition (\ref{gsh1.3}). Then for any $\frac {n}{n+1}<p\leq 1$, there holds
	$$\|T_{\phi,a}f\|_{L^p}\leq C\|f\|_{h^p}.$$
	Furthermore, if $a$ has compact support in variable $x$, then for any $0<p\leq 1$, there holds
	$$\|T_{\phi,a}f\|_{L^p}\leq C\|f\|_{h^p}.$$
\end{theoremc}
One can easily check that
$$\frac{\rho-n}{p}+(n-1)\min\{\rho,\frac 12\}>(\rho-n)(\frac 1p-\frac12)+\frac{n(\rho-\delta)}{2}$$
when $p\leq 1$ and $1-\frac{n-1}{n}\min\{\rho,1-\rho\}<\delta\leq 1.$
Therefore, Theorem C implies that one can improve Theorem A and Theorem B strictly when $\delta$ is close to 1.

We denote
\begin{equation}\label{gsh1.4}
	m_1=\rho-n+(n-1)\min\{\rho,\frac 12\}+\frac{1-\delta}{2}=\min\{n(\rho-1),\rho-\frac{n+1}{2}\}+\frac{1-\delta}{2}.
\end{equation}

In this note, when $p=1,0<\rho<1,\max\{\rho,1-\rho\}\leq\delta\leq 1$, by using some new techniques, we prove the following theorem.
\begin{theorem}
	Let $0<\rho<1,\max\{\rho,1-\rho\}\leq\delta\leq 1,a\in S^{m_1}_{\rho,\delta}$.
	Suppose that $\phi\in \Phi^{2}$ satisfies the SND condition (\ref{gsh1.3}). Then we have
	$$\|T_{\phi,a}f\|_{1}\leq C\|f\|_{h^1}.$$
	The constants here depend only on $n,\rho,\delta,\lambda$ and finitely many semi-norms $A_{N,M}, B_{N,M}$.
\end{theorem}

By using Fefferman-Stein interpolation theorem, we get the following theorem.
\begin{theorem}
	Let $\max\{\rho,1-\rho\}\leq\delta<1,a\in S^{m_p}_{\rho,\delta}$ where
	$$m_p=(m_1-\frac{n(\rho-\delta)}{2})(\frac 2p-1)+\frac{n(\rho-\delta)}{2}.$$
	Suppose that $\phi\in \Phi^{2}$ satisfies the SND condition (\ref{gsh1.3}). Then we have
	$$\|T_{\phi,a}f\|_{p}\leq C\|f\|_{p}$$
	when $1<p<2$. The constants here depend only on $n,\rho,\delta,\lambda,p$ and finitely many semi-norms $A_{N,M}, B_{N,M}$.
\end{theorem}

\textbf{Remark.} It is easy to check that
$$m_1>\frac{\rho-n}{2}+\frac{n(\rho-\delta)}{2}$$
when $\delta>\max\{\rho,1-\rho\}$. So, Theorem 1.1 and Theorem 1.2 are strict improvements of Theorem A and Theorem B when $\max\{\rho,1-\rho\}<\delta\leq 1,1\leq p<2$.
On the other hand, Theorem 1.1 coincides with Theorem B when $0<\rho<1,\delta=\max\{\rho,1-\rho\}$ and Theorem C when $0<\rho<1,\delta=1$. Therefore, Theorem 1.1 is a reasonable and ideal result.

To get Theorem 1.1, we use some new techniques. At first, we establish a completely new local $L^2$ bounded estimate. Secondly, for an $h^1$ atom supported in $B(0,r)$ when $r<1$, we directly utilize the properties of its Fourier transform. This enables us to overcome a key difficult point arising from the part $\sum\limits_{2^{j\rho}r\leq 1}$. Thirdly, we introduce a new local "exceptional set" instead of the one in Seeger-Sogge-Stein \cite{SSS91}.  Through the properties of this "exceptional set", we are able to simplify and enhance some crucial computations.

Throughout this note, $A\lesssim B$ means that $A\leq CB$ for some constant $C$. The notation $A\approx B$ means that $A\lesssim B$ and $B\lesssim A$. Without explanation,
the implicit constants given in this note may vary from occasion to occasion but depend only on $n,\rho,p,\lambda,\delta$ and finitely many semi-norms of $A_{N,M},B_{N,M}$.

\section{Some preliminaries and lemmas}

Let $B_r(x_0)$ be the ball in $\mathbb{R}^n$ centered at $x_0$ with a radius of $r$.

Firstly, we introduce the theory of atom decomposition of $h^1$.  A function $b$ is called a $L^2$-atom for $h^1(\mathbb{R}^n)$ if\\
(1) $b$ is supported in $B_r(x_0)$ for some $x_0\in \mathbb{R}^n$;\\
(2) $\|b\|_2\leq r^{-\frac n2}$;\\
(3) when $r<1$, $\int_{\mathbb{R}^n}b(x)dx=0$. When $r\geq 1$, $b$ only needs to satisfy (1) and (2).

It is also well-known that a distribution $f\in h^1(\mathbb{R}^n)$ has an atomic decomposition $f=\sum\limits_{j}\lambda_jb_j,$ with $\sum\limits_{j}|\lambda_j|<+\infty$.
Moreover, there holds
$$\|f\|_{h^1}\approx\inf\{\sum_{j}|\lambda_j|:f=\sum\limits_{j}\lambda_jb_j,\{b_j\}_{j\in \mathbb{N}}\textrm{ is a collection of } L^2\textrm{-atoms for }h^1 \}.$$

Secondly, thanks to the SND condition (\ref{gsh1.3}), for any $E\subset \mathbb{R}^n$, $\xi\in S^{n-1}$ and a real-value function $g$, there holds
\begin{align}
	\int_{\nabla_{\xi}\phi(x,\xi)\in E}\left|g\left(\nabla_{\xi}\phi(x,\xi)\right)\right|dx\lesssim \int_{E}|g(y)|dy.\label{gsh2.1}
\end{align}
This inequality will be used frequently throughout this red note.

Thirdly, we introduce the Littlewood-Paley dyadic decomposition and the second dyadic decomposition.

Take a nonnegative function $\Psi_0\in C^{\infty}_{c}(B_{2})$ such that $\Psi_0\equiv 1$ on $B_{1}$, and
set $\Psi(\xi)=\Psi_0(\xi)-\Psi_0(2\xi)$. It is easy to see that $\Psi$ is supported in $\{\xi\in \mathbb{R}^n: \frac 12<|\xi|<2\}$ and
$$\Psi_0(\xi)+\sum^{\infty}_{j=1}\Psi(2^{-j}\xi)=1, \forall \xi\in \mathbb{R}^n.$$
For the sake of convenience, we denote $\Psi_j(\xi)=\Psi(2^{-j}\xi)$ when $j>0$.

Put $\rho_0=\min\{\rho,\frac 12\}$. For every $j>0$, there are no more than $J=C2^{j(n-1)\rho_0}$ points $\xi^\nu_j\in \mathbb{S}^{n-1}$ ($\nu=1,2,\ldots, J$) such that
\begin{align*}
	|\xi^{\nu_1}_j-\xi^{\nu_2}_j|\geq 2^{-j\rho_0-2}\textrm{ if } \nu_1\neq \nu_2 \textrm{ and }
	\inf_{\nu}|\xi^{\nu}_j-\xi|\leq 2^{-j\rho_0}, \forall \xi\in \mathbb{S}^{n-1}.
\end{align*}

For $j,\nu>0$, set
\begin{align*}
	\Gamma^{\nu}_{j}=\{\xi: |\frac{\xi}{|\xi|}-\xi^{\nu}_{j}|\leq 2^{2-j\rho_0}\}\textrm{ and }B^{\nu}_{j}=\Gamma^{\nu}_{j}\cap\{\xi: 2^{j-1}<|\xi|<2^{j+1}\}.
\end{align*}
It is easy to see that
\begin{align*}
	\sum_{\nu=1}^{J}\chi_{\Gamma^{\nu}_{j}}(\xi)\leq C,\forall \xi \in S^{n-1},
\end{align*}
where $C$ depends only on $n$.

Then by using the same arguments as in \cite[p. 20-21]{CIS21}, we can construct a partition of unity $\{\psi^{\nu}_{j}\}_{\nu=1}^J$ associated with the family $\{\Gamma^{\nu}_{j}\}_{\nu=1}^J$ for any $j>0$.
Each $\psi^{\nu}_{j}$ is positively homogeneous of degree 0, supported in $\Gamma^{\nu}_{j}$, and satisfies that
\begin{align}
	\sum^{J}_{\nu=1}\psi^{\nu}_{j}(\xi)\equiv1 \textrm{ if }\xi\neq 0,\left|\nabla^{k}_{\xi}\psi^{\nu}_{j}(\xi)\right|\leq C_{k}|\xi|^{-k}2^{jk\rho_0}, k\in\mathbb{N}, \label{gsh2.2}
\end{align}
where $C_k$ depends only on $k$.

Next, we usually use the following notations,
\begin{align*}
	&a_j(x,\xi)=a(x,\xi)\Psi_j(\xi),\hspace{0.3cm}T_{\phi,a_j}f(x)=\int_{\mathbb{R}^{n}}e^{i\phi(x,\xi)}a_j(x,\xi)\widehat{f}(\xi)d\xi,\hspace{0.3cm} j\geq 0;\\
	&T^{\nu}_{j}f(x)=\int_{\mathbb{R}^{n}}e^{i\phi(x,\xi)}a_j(x,\xi)\psi^{\nu}_{j}(\xi)\widehat{f}(\xi)d\xi,\hspace{0.3cm} j,\hspace{0.1cm}\nu>0.
\end{align*}
One can easily see that $T_{\phi,a_j}f,T^{\nu}_{j}f$ are well-defined for any $f\in L^{1}$.

Set
$$h^{\nu}_{j}(x,\xi)=\phi(x,\xi)-\xi\cdot \nabla_{\xi}\phi(x,\xi^{\nu}_{j}).$$
By using similar arguments in \cite[p. 407]{S93} or \cite{CIS21,DGZ23,FS14,SZ23}, we get the following lemma.
\begin{lemma}
	If $j>0$ and $\phi\in \Phi^{2}$, then for any $\xi\in B^{\nu}_{j}$ we have
	\begin{align*}
		&|\partial^{N}_{\xi^{\nu}_{j}}\nabla^{M}_{\xi}\nabla^{K}_{x}h^{\nu}_{j}|\lesssim 2^{-j(N\rho+M\rho_0)}, \textrm{if $N,M,K\geq 0$ with $N+M\geq 1$};\\
		&|\partial^{N}_{\xi^{\nu}_{j}}\nabla^{M}_{\xi}\psi^{\nu}_{j}|\lesssim 2^{-j(N+M(1-\rho_0))}\leq 2^{-j(N\rho+M\rho_0)}, \textrm{if $N,M\geq 0$}.
	\end{align*}
	The constants depend only on $n,\rho,N,M,K$ and finitely many semi-norms of $\phi\in \Phi^{2}$.
\end{lemma}
It is easy to see that $\partial^{\alpha}_x\phi\in \Phi^{2}$ for any $\alpha\in\mathbb{N}^n$.
So by the rotation invariance, it is sufficient for us to prove this lemma only for $K=0$ and $\xi^{\nu}_{j}=(1,0,\cdots,0)$.
Then one can see the details in the proof of \cite[Lemma 3.1]{DGZ23}.

As a direct corollary, we give the following lemma.
\begin{lemma}
	If $N,M,K\geq 0,j>0$, $a\in S^{m}_{\rho,\delta}$ and $\phi\in \Phi^{2}$, then for any $\xi\in B^{\nu}_{j}$ there holds
	$$\left|\partial^{N}_{\xi^{\nu}_{j}}\nabla^{M}_{\xi}\nabla^{K}_{x}(e^{ih^{\nu}_{j}}a_j\psi^{\nu}_{j})\right|\leq C2^{j(m-N\rho-M\rho_0+K\delta)},$$
	where $C$ depends only on $n,m,\rho,N,M,K$ and finitely many semi-norms of $a\in S^{m}_{\rho,\delta}$ and $\phi\in \Phi^{2}$.
\end{lemma}
\begin{proof}
	By Lemma 2.1 and the product rule, for any $j>0,k,l,s\geq 0$ with $k+l+s>0$, we can get that
	\begin{align*}
		|\partial^{k}_{\xi^{\nu}_{j}}\nabla^{l}_{\xi}\nabla^{s}_{x}e^{ih^{\nu}_{j}}|
		\lesssim &\sum^{k+l+s}_{t=1}\sum_{\begin{subarray}{c}k_1+\cdots+k_t=k,l_1+\cdots+l_t=l,s_1+\cdots+s_t=s\\k_1+l_1+s_1,\ldots,k_t+l_t+s_t>0\end{subarray}}
		|\partial^{k_1}_{\xi^{\nu}_{j}}\nabla^{l_1}_{\xi}\nabla^{s_1}_{x}h^{\nu}_{j}|\cdots |\partial^{k_t}_{\xi^{\nu}_{j}}\nabla^{l_t}_{\xi}\nabla^{s_t}_{x}h^{\nu}_{j}|\nonumber\\
		\lesssim &\sum^{k+l+s}_{t=1}\sum_{\begin{subarray}{c}k_1+\cdots+k_t=k,l_1+\cdots+l_t=l,s_1+\cdots+s_t=s\\k_1+l_1+s_1,\ldots,k_t+l_t+s_t>0\end{subarray}}
		2^{-j(k_1\rho+l_1\rho_0)}\cdots 2^{-j(k_t\rho+l_t\rho_0)}\nonumber\\
		\lesssim &2^{-j(k\rho+l\rho_0)}.
	\end{align*}
	This inequality is certainly true for $k=l=s=0$.
	
	Similarly, since $a\in S^{m}_{\rho,\delta}$ and $\rho_0\leq \rho$, for any $j>0,k,l,s\geq 0$, by Lemma 2.1 we get that
	\begin{align*}
		|\partial^{k}_{\xi^{\nu}_{j}}\nabla^{l}_{\xi}\nabla^{s}_{x}(a_j\psi^{\nu}_{j})|
		\lesssim &\sum_{k_1+k_2=k}\sum_{l_1+l_2=l}
		|\nabla^{k_1+l_1}_{\xi}\nabla^{s}_{x}a_j| |\partial^{k_2}_{\xi^{\nu}_{j}}\nabla^{l_2}_{\xi}\psi^{\nu}_{j}|\nonumber\\
		\lesssim &\sum_{k_1+k_2=k}\sum_{l_1+l_2=l}2^{j(m-(k_1+l_1)\rho+s\delta)}2^{-j(k_2\rho+l_2\rho_0)}\lesssim 2^{j(m-k\rho-l\rho_0+s\delta)}.
	\end{align*}
	Therefore,  we show that
	\begin{align*}
		|\partial^{N}_{\xi^{\nu}_{j}}\nabla^{M}_{\xi}\nabla^{K}_{x}(e^{ih^{\nu}_{j}}a_j\psi^{\nu}_{j})|
		\lesssim & \sum_{N_1+N_2=N}\sum_{M_1+M_2=M}\sum_{K_1+K_2=K}|\partial^{N_1}_{\xi^{\nu}_{j}}\nabla^{M_1}_{\xi}\nabla^{K_1}_{x}e^{ih^{\nu}_{j}}||\partial^{N_2}_{\xi^{\nu}_{j}}\nabla^{M_2}_{\xi}\nabla^{K_2}_{x}(a_j\psi^{\nu}_{j})|\\
		\lesssim &\sum_{N_1+N_2=N}\sum_{M_1+M_2=M}\sum_{K_1+K_2=K}2^{-j(N_1\rho+M_1\rho_0)}2^{j(m-N_2\rho-M_2\rho_0+K_2\delta)}\\
		\lesssim &2^{j(m-N\rho-M\rho_0+K\delta)}.
	\end{align*}
	This finishes the proof.
\end{proof}

At last, we introduce a local "exceptional set" and obtain some basic properties.

Set $\lambda_{\rho}=1$ when $0<\rho\leq \frac 12$ and $\lambda_{\rho}=\frac{1}{2\rho}$ when $\frac 12<\rho\leq 1$. It is easy to see that $\rho_0=\lambda_{\rho}\rho$. For $r>0$, we set
$$R^{\nu}_{j,r}=\{x\in\mathbb{R}^n: |\xi^{\nu}_{j}\cdot(\nabla_{\xi}\phi)(x,\xi^{\nu}_{j})|\leq 3r \textrm{ and } |(\nabla_{\xi}\phi)(x,\xi^{\nu}_{j})|\leq 3r^{\lambda_{\rho}}\}.$$
Due to (\ref{gsh2.1}), it is easy to check that
$$|R^{\nu}_{j,r}|\lesssim |\{z\in\mathbb{R}^n:|\xi^{\nu}_{j}\cdot z|\leq 3r \textrm{ and } |z|\leq 3r^{\lambda_{\rho}}\}|\lesssim r^{1+(n-1)\lambda_{\rho}}.$$

When $j>0$, for any $x\notin R^{\nu}_{j,r},|y|<r$, we claim that
\begin{align}
	&1+2^{j\rho}|\xi^{\nu}_{j}\cdot(\nabla_{\xi}\phi)(x,\xi^{\nu}_{j})|+2^{j\rho_0}|(\nabla_{\xi}\phi)(x,\xi^{\nu}_{j})|\nonumber\\
	\lesssim &(1+2^{j\rho}|\xi^{\nu}_{j}\cdot((\nabla_{\xi}\phi)(x,\xi^{\nu}_{j})-y)|+2^{j\rho_0}|(\nabla_{\xi}\phi)(x,\xi^{\nu}_{j})-y|)^{\frac{1}{\lambda_{\rho}}}.\label{gsh2.5}
\end{align}

It is easy to see that $1+2^{j\rho}r+2^{j\rho_0}r^{\lambda_{\rho}}\leq 2(1+2^{j\rho}r)$ as $\lambda_{\rho}\leq 1$ and $\rho_0=\lambda_{\rho}\rho$. So, when $|y|<r$, there holds
\begin{align*}
	&1+2^{j\rho}|\xi^{\nu}_{j}\cdot(\nabla_{\xi}\phi)(x,\xi^{\nu}_{j})|+2^{j\rho_0}|(\nabla_{\xi}\phi)(x,\xi^{\nu}_{j})|\\
	\leq &1+2^{j\rho}|\xi^{\nu}_{j}\cdot((\nabla_{\xi}\phi)(x,\xi^{\nu}_{j})-y)|+2^{j\rho_0}|(\nabla_{\xi}\phi)(x,\xi^{\nu}_{j})-y|
	+2^{j\rho}|\xi^{\nu}_{j}\cdot y|+2^{j\rho_0}|y|\\
	\leq &1+2^{j\rho}|\xi^{\nu}_{j}\cdot((\nabla_{\xi}\phi)(x,\xi^{\nu}_{j})-y)|+2^{j\rho_0}|(\nabla_{\xi}\phi)(x,\xi^{\nu}_{j})-y|+2(1+2^{j\rho}r).
\end{align*}
Therefore, to prove (\ref{gsh2.5}), it is sufficient  for us to show that
$$2^{j\rho}r\lesssim (1+2^{j\rho}|\xi^{\nu}_{j}\cdot((\nabla_{\xi}\phi)(x,\xi^{\nu}_{j})-y)|+2^{j\rho_0}|(\nabla_{\xi}\phi)(x,\xi^{\nu}_{j})-y|)^{\frac{1}{\lambda_{\rho}}}.$$
For any $x\notin R^{\nu}_{j,r}$, there must be
$$|\xi^{\nu}_{j}\cdot(\nabla_{\xi}\phi)(x,\xi^{\nu}_{j})|>3r\textrm{ or }|(\nabla_{\xi}\phi)(x,\xi^{\nu}_{j})|>3r^{\lambda_{\rho}}.$$
If $|\xi^{\nu}_{j}\cdot(\nabla_{\xi}\phi)(x,\xi^{\nu}_{j})|>3r$, due to the fact $\lambda_{\rho}\leq 1$, we have
$$(1+2^{j\rho}|\xi^{\nu}_{j}\cdot((\nabla_{\xi}\phi)(x,\xi^{\nu}_{j})-y)|+2^{j\rho_0}|(\nabla_{\xi}\phi)(x,\xi^{\nu}_{j})-y|)^{\frac{1}{\lambda_{\rho}}}\geq (1+2^{j\rho}r)^{\frac{1}{\lambda_{\rho}}}\geq 2^{j\rho}r.$$
Otherwise, if $|(\nabla_{\xi}\phi)(x,\xi^{\nu}_{j})|>3r^{\lambda_{\rho}}$, as $\lambda_{\rho}\leq 1$ and $\rho_0=\lambda_{\rho}\rho$, we get that
\begin{align*}
	(1+2^{j\rho}|\xi^{\nu}_{j}\cdot((\nabla_{\xi}\phi)(x,\xi^{\nu}_{j})-y)|+2^{j\rho_0}|(\nabla_{\xi}\phi)(x,\xi^{\nu}_{j})-y|)^{\frac{1}{\lambda_{\rho}}}\geq (1+2^{j\rho_0}r^{\lambda_{\rho}})^{\frac{1}{\lambda_{\rho}}}\geq
	2^{j\rho}r.
\end{align*}
This finishes the proof of (\ref{gsh2.5}).

\section{A local $L^2$ bounded estimate}

 At first we introduce a continuous version of the Cotlar-Stein lemma proved in Calder\'{o}n-Vaillancourt \cite{CV71}.
\begin{lemma}\label{al2}
	(\cite{CV71})
	Let $A_{\xi}$ be a $\xi$-weakly measurable and uniformly bounded family of operators on $L^{2},$ $\Vert A_{\xi}\Vert\leq M_{0}$ for all $\xi$ in a measure space $E$.
	If the inequality
	$$\Vert A_{\xi}A_{\eta}^{*}\Vert+\Vert A^{*}_{\xi}A_{\eta}\Vert\leq h^2(\xi,\eta)$$
	holds with a nonnegative function $h(\xi,\eta)$ which is the kernel of a bounded integral operator $H$ on $L^{2}$ with norm $M$, then the operator
	$$A=\int_{E} A_{\xi}d\xi$$
	is bounded on $L^{2}$ with norm $\Vert A\Vert\leq M.$
\end{lemma}

When $0\leq \rho\leq 1,\max\{1-\rho,\rho\}\leq \delta\leq 1$, we establish the following technical lemma which plays an important role in this article.
\begin{lemma}
	When $0\leq \rho\leq 1,\max\{1-\rho,\rho\}\leq \delta\leq 1, j>0$, suppose that $\tilde{a}$ is supported in $\mathbb{R}^n\times B^{\nu}_{j}$ and satisfies
	\begin{align*}
		\sup_{x,\xi\in\mathbb{R}^{n}}|\partial^{N_1}_{\xi^{\nu}_{j}}\nabla^{N_2}_{\xi}\nabla^{M}_{x}\tilde{a}(x,\xi)|\leq A_{N_1,N_2,M}2^{j(m+\rho N_1+\rho_0N_2-\delta M)}
	\end{align*}
	for any $N_1,N_2,M\in \mathbb{N}$ and some positive constants $A_{N_1,N_2,M}$. Then the pseudo-differential operator $T_{\tilde{a}}$ is bounded on $L^2$, i.e.
	\begin{align*}
		\|T_{\tilde{a}}f\|_2\leq C2^{j(m+(n-1)(\frac 12-\rho_0)+\frac {\delta-\rho}{2})}\|f\|_2,
	\end{align*}
	where $C$ depends only on $n,\rho,\delta,m$ and finitely many semi-norms $A_{N_1,N_2,M}$.
\end{lemma}
\begin{proof}  Set
	$$S_{\tilde{a}}f(x)=\int_{\mathbb{R}^{n}}e^{i x\cdot\xi}\tilde{a}(x,\xi)f(\xi)d\xi.$$
	It is easy to see that $S_{\tilde{a}}S^*_{\tilde{a}}$ can be given by
	$$S_{\tilde{a}}S^*_{\tilde{a}}f(x)=\int_{\mathbb{R}^n}\int_{\mathbb{R}^n}e^{i(x-y)\cdot\xi}\tilde{a}(x,\xi)\overline{\tilde{a}(y,\xi)}d\xi f(y)dy.$$
	Take an integer $N>n$ and set
	$$\mu_{j,N}(x)=(1+2^{j\rho}|\xi^{\nu}_{j}\cdot x|+2^{j\rho_0}|x|)^{-2N}.$$
	Due to the assumptions on $\tilde{a}$ and integration by parts, we can obtain that
	\begin{align*}
		&\int_{\mathbb{R}^n}e^{i(x-y)\cdot\xi}\tilde{a}(x,\xi)\overline{\tilde{a}(y,\xi)}d\xi\\
		=&(1+2^{j\rho}|\xi^{\nu}_{j}\cdot(x-y)|+2^{j\rho_0}|x-y|)^{-2N}\\
		&\int_{\mathbb{R}^n}e^{i(x-y)\cdot\xi}\sum_{N_1+|\beta|\leq 2N}
		C_{N_1,\beta}2^{j(\rho N_1+\rho_0 |\beta|)}\partial^{N_1}_{\xi^{\nu}_{j}}\partial^{\beta}_{\xi}(\tilde{a}(x,\xi)\overline{\tilde{a}(y,\xi)})d\xi.
	\end{align*}
	Thus, we can write
	$$S_{\tilde{a}}S^*_{\tilde{a}}f(x)=\int_{\mathbb{R}^n}\int_{\mathbb{R}^n}e^{i(x-y)\cdot\xi}\tilde{b}(x,y,\xi)d\xi f(y)dy$$
	where
	$$\tilde{b}(x,y,\xi)=\mu_{j,N}(x-y)\sum_{N_1+|\beta|\leq 2N}
	C_{N_1,\beta}2^{j(\rho N_1+\rho_0 |\beta|)}\partial^{N_1}_{\xi^{\nu}_{j}}\partial^{\beta}_{\xi}(\tilde{a}(x,\xi)\overline{\tilde{a}(y,\xi)})$$
	is supported in $\mathbb{R}^n\times\mathbb{R}^n\times B^{\nu}_{j}$. As $\rho_0\leq \rho\leq \delta$, we can see that
	\begin{align}\label{gs11}
		\sup_{x,y,\xi\in\mathbb{R}^{n}}|\nabla^{K}_{x}\nabla^{M_2}_{y}\tilde{b}(x,y,\xi)|\leq C_{N,K,M_2}2^{j(2m+\delta (K+M_2))}\mu_{j,N}(x-y)
	\end{align}
	for any $N,K,M_2\in \mathbb{N}$ and some positive constants $C_{N,K,M_2}$.
	
	Set
	$$A_{\xi}f(x)=\int_{\mathbb{R}^n}e^{i(x-y)\cdot\xi}\tilde{b}(x,y,\xi)f(y)dy.$$
	As $b$ is supported in $\mathbb{R}^n\times\mathbb{R}^n\times B^{\nu}_{j}$, we have
	$$S_{\tilde{a}}S^*_{\tilde{a}}f(x)=\int_{B^{\nu}_{j}}A_{\xi}f(x)d\xi.$$
	Now, for any $\xi,\eta\in B^{\nu}_{j}$, we estimate the norm $\Vert A_{\xi}A_{\eta}^{*}\Vert_{L^2-L^2}+\Vert A^{*}_{\xi}A_{\eta}\Vert_{L^2-L^2}$.
	From the definition of $A_{\xi}$, some simple computations yield that
	\begin{align*}
		A_{\xi}A_{\eta}^{*}f(x)=&\int_{\mathbb{R}^n}e^{i(x-z)\cdot\xi}\tilde{b}(x,z,\xi)A_{\eta}^{*}f(z)dz\\
		=&\int_{\mathbb{R}^n}\int_{\mathbb{R}^n}e^{i(x-z)\cdot\xi}e^{-i(y-z)\cdot\eta}\tilde{b}(x,z,\xi)\overline{\tilde{b}(y,z,\eta)}f(y)dzdy\\
		=&\int_{\mathbb{R}^n}e^{i(x\cdot\xi-y\cdot\eta)}\int_{\mathbb{R}^n}e^{iz\cdot(\eta-\xi)}\tilde{b}(x,z,\xi)\overline{\tilde{b}(y,z,\eta)}dzf(y)dy.
	\end{align*}
	By (\ref{gs11}), we can get that
	\begin{align}\label{gs12}
		\sup_{x,\xi\in\mathbb{R}^{n}}|\nabla^{N}_{z}[\tilde{b}(x,z,\xi)\overline{\tilde{b}(y,z,\eta)}]|
		\leq C_{N}2^{j(4m+\delta N)}\mu_{j,N}(x-z)\mu_{j,N}(y-z).
	\end{align}
	By integration by parts and (\ref{gs12}), we obtain that
	\begin{align}
		&|\int_{\mathbb{R}^n}e^{iz\cdot(\eta-\xi)}\tilde{b}(x,z,\xi)\overline{\tilde{b}(y,z,\eta)}dz|\nonumber\\
		\leq &C_{N}(1+2^{-j\delta}|\eta-\xi|)^{-2N}\sum_{|\alpha|\leq 2N}|\int_{\mathbb{R}^n}e^{iz\cdot(\eta-\xi)}2^{-j\delta|\alpha|}\partial^{\alpha}_{z}[\tilde{b}(x,z,\xi)\overline{\tilde{b}(y,z,\eta)}]dz|\nonumber\\
		\leq &C_{N}2^{4jm}(1+2^{-j\delta}|\eta-\xi|)^{-2N}\int_{\mathbb{R}^n}\mu_{j,N}(x-z)\mu_{j,N}(y-z)dz.\label{gs13}
	\end{align}
	Thus, we have
	$$|A_{\xi}A_{\eta}^{*}f(x)|\leq C_{N}2^{4jm}(1+2^{-j\delta}|\eta-\xi|)^{-2N}\int_{\mathbb{R}^n}\int_{\mathbb{R}^n}\mu_{j,N}(x-z)\mu_{j,N}(y-z)|f(y)|dzdy$$
	Then, Young's inequality implies that
	\begin{align}
		\|A_{\xi}A_{\eta}^{*}f\|_2\leq &C_{N}2^{4jm}(1+2^{-j\delta}|\eta-\xi|)^{-2N}\|\mu_{j,N}\|_1\|\mu_{j,N}\|_1\|f\|_2\nonumber\\
		\leq &C_{N}2^{2j(2m-\rho-(n-1)\rho_0)}(1+2^{-j\delta}|\eta-\xi|)^{-2N}\|f\|_2.\label{gs14}
	\end{align}
	The same computations yield that
	\begin{align*}
		\|A^{*}_{\xi}A_{\eta}f\|_2\leq &C_{N}2^{2j(2m-\rho-(n-1)\rho_0)}(1+2^{-j\delta}|\eta-\xi|)^{-2N}\|f\|_2.
	\end{align*}
	Set
	$$h(\xi,\eta)=2^{j(2m-\rho-(n-1)\rho_0)}(1+2^{-j\delta}|\eta-\xi|)^{-N}\chi_{B^{\nu}_{j}}(\xi)\chi_{B^{\nu}_{j}}(\eta).$$
	Then we have show that
	$$\|A_{\xi}A_{\eta}^{*}\|_{L^2\rightarrow L^2}+\|A^{*}_{\xi}A_{\eta}\|_{L^2\rightarrow L^2}\leq C_{N}h^2(\xi,\eta).$$
	As $\xi,\eta\in B^{\nu}_{j}$ and $1-\rho_0\leq \max\{1-\rho,\rho\}\leq \delta$, it is easy to check that
	$$2^{-j\delta}|(\eta-\xi)-(\xi^{\nu}_{j}\cdot(\eta-\xi))\xi^{\nu}_{j}|\leq C2^{-j\delta}2^{j(1-\rho_0)}\leq C.$$
	Therefore, for any $\xi\in B^{\nu}_{j}$, we can obtain that
	\begin{align*}
		\int_{B^{\nu}_{j}}h(\xi,\eta)d\eta\leq &2^{j(2m-\rho-(n-1)\rho_0)}\int_{B^{\nu}_{j}}(1+2^{-j\delta}|\eta-\xi|)^{-N}d\eta\\
		\leq &2^{j(2m-\rho-(n-1)\rho_0)}\int_{B^{\nu}_{j}}(1+2^{-j\delta}|\xi^{\nu}_{j}\cdot(\eta-\xi)|)^{-N}d\eta\\
		\leq &C_N2^{j[2m-\rho-(n-1)\rho_0+\delta+(n-1)(1-\rho_0)]}\\
		= &C_N2^{j[2m+(n-1)(1-2\rho_0)+\delta-\rho]}
	\end{align*}
	and
	$$\int_{B^{\nu}_{j}}h(\xi,\eta)d\xi\leq C2^{j[2m+(n-1)(1-2\rho_0)+\delta-\rho]}.$$
	
	Then, by Schur's lemma, the integral operator $B$ given by $B(f)(\xi)=\int_{B^{\nu}_{j}}h(\xi,\eta)f(\eta)d\eta$ is bounded on $L^2$ and
	\begin{align*}
		\|B\|_{L^2\rightarrow L^2}\leq \biggl(\sup_{\zeta\in B^{\nu}_{j}}\int_{B^{\nu}_{j}}h(\zeta,\eta)d\eta \sup_{\zeta\in B^{\nu}_{j}}\int_{B^{\nu}_{j}}h(\xi,\zeta)d\xi\biggl)^{1/2}\leq C2^{j[2m+(n-1)(1-2\rho_0)+\delta-\rho]}.
	\end{align*}
	
	Hence, by Lemma \ref{al2}, we get that
	$$\|S_{\tilde{a}}S^*_{\tilde{a}}f\|_2\leq C2^{j[2m+(n-1)(1-2\rho_0)+\delta-\rho]}.$$
	By a standard dual argument, we have
	$$\|S_{\tilde{a}}\|_{L^2\rightarrow L^2}=\|S_{\tilde{a}}S^*_{\tilde{a}}\|^{\frac 12}_{L^2\rightarrow L^2}\lesssim 2^{j[m+(n-1)(\frac 12-\rho_0)+\frac{\delta-\rho}{2}]}.$$
	It follows from the Plancherel theorem that
	\begin{align*}
		\|T_{\tilde{a}}f\|_2=\|S_{\tilde{a}}\widehat{f}\|_2\lesssim 2^{j[m+(n-1)(\frac 12-\rho_0)+\frac{\delta-\rho}{2}]}\|\widehat{f}\|_2=2^{j[m+(n-1)(\frac 12-\rho_0)+\frac{\delta-\rho}{2}]}\|f\|_2.
	\end{align*}
	This finishes the proof.
\end{proof}

As a corollary, we can obtain the following lemma.
\begin{lemma}
	Under the conditions in Theorem 1.1, we can get that
	\begin{align*}
		\|T^{\nu}_{j}f\|_2\lesssim 2^{\frac{j(\rho-n)}{2}}\|f\|_2.
	\end{align*}
\end{lemma}
\begin{proof} Thanks to the SND condition (\ref{gsh1.3}), we can define the inverse mapping of $x\to \nabla_{\xi}\phi(x,\xi^{\nu}_{j})$ as $F^{\nu}_{j}$, i.e.
	$$F^{\nu}_{j}(\nabla_{\xi}\phi(x,\xi^{\nu}_{j}))=x, \hspace{0.2cm} \forall x\in \mathbb{R}^n.$$
	From (\ref{gsh1.2}) and (\ref{gsh1.3}) we can get that
	\begin{align}\label{inv}
		|\nabla^N F^{\nu}_{j}(x)|\leq C_N
	\end{align}
	for any $N\in \mathbb{N}$, where $C_N$ depends only on $N,\lambda$ in (\ref{gsh1.3}) and finitely many semi-norms in (\ref{gsh1.2}).
	
	Set
	$$\tilde{a}^{\nu}_{j}(x,\xi)=e^{ih^{\nu}_{j}(F^{\nu}_{j}(x),\xi)}a_j(F^{\nu}_{j}(x),\xi)\psi^{\nu}_{j}(\xi), \hspace{0.2cm}\forall x\in \mathbb{R}^n.$$
	From Lemma 2.2 and (\ref{inv}), one can see that
	\begin{align*}
		\sup_{x,\xi\in\mathbb{R}^{n}}|\partial^{N_1}_{\xi^{\nu}_{j}}\nabla^{N_2}_{\xi}\nabla^{M}_{x}\tilde{a}^{\nu}_{j}(x,\xi)|\leq C_{N_1,N_2,M}2^{j(m_1-\rho N_1-\rho_0N_2+\delta M)}.
	\end{align*}
	It is easy to see that $\tilde{a}^{\nu}_{j}$ is supported in $\mathbb{R}^n\times B^{\nu}_{j}$. Therefore, by Lemma 3.2, we get that
	\begin{align}\label{pse}
		\|T_{\tilde{a}^{\nu}_{j}}f\|_2\lesssim2^{j(m_1+(n-1)(\frac 12-\rho_0)+\frac {\delta-\rho}{2})}\|f\|_2=2^{\frac{j(\rho-n)}{2}}\|f\|_2.
	\end{align}
	On the other hand, it is to see that
	\begin{align*}
		T^{\nu}_{j}f(x)=&\int_{\mathbb{R}^{n}}e^{i\phi(x,\xi)}a_j(x,\xi)\psi^{\nu}_{j}(\xi)\widehat{f}(\xi)d\xi\\
		=&\int_{\mathbb{R}^{n}}e^{i\nabla_{\xi}\phi(x,\xi^{\nu}_{j})\cdot \xi}[e^{ih^{\nu}_{j}}a_j\psi^{\nu}_{j}](x,\xi)\widehat{f}(\xi)d\xi\\
		=&\int_{\mathbb{R}^{n}}e^{i\nabla_{\xi}\phi(x,\xi^{\nu}_{j})\cdot \xi}[e^{ih^{\nu}_{j}}a_j\psi^{\nu}_{j}](F^{\nu}_{j}(\nabla_{\xi}\phi(x,\xi^{\nu}_{j})),\xi)\widehat{f}(\xi)d\xi\\
		=&T_{\tilde{a}^{\nu}_{j}}f(\nabla_{\xi}\phi(x,\xi^{\nu}_{j})).
	\end{align*}
	By inequalities (\ref{gsh2.1}) and (\ref{pse}), we have
	\begin{align*}
		\|T^{\nu}_{j}f\|_2=&\biggl(\int_{\mathbb{R}^{n}}|T_{\tilde{a}^{\nu}_{j}}f(\nabla_{\xi}\phi(x,\xi^{\nu}_{j}))|^2dx\biggl)^{1/2}\\
		\lesssim &\|T_{\tilde{a}^{\nu}_{j}}f\|_2\lesssim 2^{\frac{j(\rho-n)}{2}}\|f\|_2.
	\end{align*}
	This finishes the proof.
\end{proof}

\section{Proof of Theorem 1.1}

 Due to the theory of atom decomposition, it suffices to show $\|T_{\phi,a} b\|_1\lesssim 1$ for any $L^2$-atom $b$ for $h^1(\mathbb{R}^n)$ which is defined in Section 2. By the translation-invariance,
we can assume that the center of atom $b$ is at the origin.We divide $\|T_{\phi,a} b\|_1$ into three parts,
$$\|T_{\phi,a} b\|_1\leq \|T_{\phi,a_0} b\|_1+\sum\limits_{2^{j\rho}r<1,j>0}\|T_{\phi,a_j}b\|_1+\sum\limits_{2^{j\rho}r\geq 1,j>0}\|T_{\phi,a_j}b\|_1:=I+II+III.$$
In fact, \cite[Theorem 1.18]{FS14} gives that $I \lesssim \|b\|_1\lesssim 1$. Next we estimate them separately (the second part is empty if $r\geq 1$).

\subsection{Part II: $\sum\limits_{2^{j\rho}r<1,j>0}\|T_{\phi,a_j}b\|_1$}

 In this case, there must be $r<1$. For any $N\in \mathbb{N}$, there holds
\begin{align}
	|\nabla^N\widehat{b}(\xi)|\lesssim \int_{|y|\leq r}|y|^N|b(y)|dy\lesssim r^N.\label{gsh3.4}
\end{align}
On the other hand, due to $\int_{|y|\leq r}b(y)dy=0$ when $r<1$, it is easy to see that
$$|\widehat{b}(\xi)|=|\int_{|y|\leq r}(e^{i2\pi y\cdot \xi}-1)b(y)dy|\lesssim r|\xi|.$$
From the definition of $B^{\nu}_{j}$ and the Plancherel theorem, we have
\begin{align*}
	\|\sum^{J}_{\nu=1}\chi_{B^{\nu}_{j}}|\widehat{b}|\|_2\lesssim (\int_{2^{j-1}<|\xi|<2^{j+1}}|\widehat{b}(\xi)|^2d\xi)^{\frac 12}\lesssim \min\{\|b\|_2,2^{j(1+\frac n2)}r\}
	\lesssim \min\{r^{-\frac n2},2^{j(1+\frac n2)}r\}
\end{align*}
which yields that
\begin{align}
	\sum^{\infty}_{j=0}2^{-\frac {jn}{2}}\|\sum^{J}_{\nu=1}\chi_{B^{\nu}_{j}}|\widehat{b}|\|_2\lesssim  \sum^{\infty}_{j=0}\min\{(2^jr)^{-\frac n2},2^jr\}\lesssim 1.\label{gsh3.3}
\end{align}

Set
$$L=1-2^{2j\rho}\partial^2_{\xi^{\nu}_{j}}-2^{2j\rho_0}\sum\limits^n_{k=1}\partial^2_{\xi_k}.$$
Obviously, $L$ is self-adjoint. Then as $\rho_0\leq \rho$, for any $N\in \mathbb{N}$ we can show that
\begin{align*}
	&|T^{\nu}_{j}b(x)|=|\int_{\mathbb{R}^{n}}e^{i\nabla_{\xi}\phi(x,\xi^{\nu}_j)\cdot \xi}e^{ih_j^{\nu}(x,\xi)} a_j(x,\xi)\psi^{\nu}_{j}(\xi)\widehat{b}(\xi)d\xi|\\
	=&(1+2^{2j\rho}|\partial_{\xi^{\nu}_{j}}\phi(x,\xi^{\nu}_j)|^2+2^{2j\rho_0}|\nabla_{\xi}\phi(x,\xi^{\nu}_j)|^2)^{-N}\\
	&|\int_{\mathbb{R}^{n}}L^N(e^{i\nabla_{\xi}\phi(x,\xi^{\nu}_j)\cdot \xi})e^{ih_j^{\nu}(x,\xi)} a_j(x,\xi)\psi^{\nu}_{j}(\xi)\widehat{b}(\xi)d\xi|\\
	=&(1+2^{2j\rho}|\partial_{\xi^{\nu}_{j}}\phi(x,\xi^{\nu}_j)|^2+2^{2j\rho_0}|\nabla_{\xi}\phi(x,\xi^{\nu}_j)|^2)^{-N}\\
	&|\int_{\mathbb{R}^{n}}e^{i\nabla_{\xi}\phi(x,\xi^{\nu}_j)\cdot \xi}L^N(e^{ih_j^{\nu}(x,\xi)} a_j(x,\xi)\psi^{\nu}_{j}(\xi)\widehat{b}(\xi))d\xi|\\
	\lesssim &(1+2^{j\rho}|\partial_{\xi^{\nu}_{j}}\phi(x,\xi^{\nu}_j)|+2^{j\rho_0}|\nabla_{\xi}\phi(x,\xi^{\nu}_j)|)^{-2N}\\
	&\sum_{k_1+|\alpha|+|\beta|\leq 2N}|\int_{\mathbb{R}^{n}}e^{i\nabla_{\xi}\phi(x,\xi^{\nu}_j)\cdot \xi}
	2^{jk_1\rho}2^{j|\alpha|\rho_0}\partial^{k_1}_{\xi^{\nu}_{j}}\partial^{\alpha}_{\xi}(e^{ih_j^{\nu}}a_j\psi^{\nu}_{j})
	2^{j|\beta|\rho}\partial^{\beta}_{\xi}\widehat{b}d\xi|\\
	= &(1+2^{j\rho}|\partial_{\xi^{\nu}_{j}}\phi(x,\xi^{\nu}_j)|+2^{j\rho_0}|\nabla_{\xi}\phi(x,\xi^{\nu}_j)|)^{-2N}
	\sum_{k_1+|\alpha|+|\beta|\leq 2N}|T_{j,k_1,\alpha}(2^{j|\beta|\rho}(\chi_{B^{\nu}_{j}}\partial^{\beta}_{\xi}\widehat{b})^{\vee})(x)|,
\end{align*}
where
\begin{equation*}
	T_{j,k_1,\alpha}f(x)=\int_{\mathbb{R}^{n}}e^{i\nabla_{\xi}\phi(x,\xi^{\nu}_j)\cdot \xi}
	2^{jk_1\rho}2^{j|\alpha|\rho_0}\partial^{k_1}_{\xi^{\nu}_{j}}\partial^{\alpha}_{\xi}(e^{ih_j^{\nu}}a_j\psi^{\nu}_{j})
	\widehat{f}d\xi.
\end{equation*}

From Lemma 2.2 and Lemma 3.2, we can use the same proof of Lemma 3.3 to get
\begin{align*}
	\|T_{j,k_1,\alpha}f\|_2\leq C2^{\frac{j(\rho-n)}{2}}\|f\|_2,
\end{align*}
where $C$ depends only on $n,\rho,\delta,\lambda,N$ finitely many semi-norms of $\phi\in \Phi^2,a\in S^{m_1}_{\rho,\delta}$.

Take an integer $N>n/2$. From (\ref{gsh2.1}), (\ref{gsh3.4}), (\ref{gsh3.3}) and the Plancherel theorem, we have
\begin{align}
	&\sum_{2^{j\rho}r<1,j>0}\sum^{J}_{\nu=1}\|T^{\nu}_{j}b\|_1\nonumber\\
	\lesssim &\sum_{2^{j\rho}r<1,j>0}\sum^{J}_{\nu=1}\sum_{k_1+|\alpha|+|\beta|\leq 2N}\nonumber\\
	&\int_{\mathbb{R}^n}(1+2^{j\rho}|\partial_{\xi^{\nu}_{j}}\phi(x,\xi^{\nu}_j)|+2^{j\rho_0}|\nabla_{\xi}\phi(x,\xi^{\nu}_j)|)^{-2N}
	|T_{j,k_1,\alpha}(2^{j|\beta|\rho}(\chi_{B^{\nu}_{j}}\partial^{\beta}_{\xi}\widehat{b})^{\vee})(x)|dx\nonumber\\
	\lesssim &\sum_{2^{j\rho}r<1,j>0}\sum^{J}_{\nu=1}\sum_{k_1+|\alpha|+|\beta|\leq 2N}\nonumber\\
	&\biggl(\int_{\mathbb{R}^n}(1+2^{j\rho}|\partial_{\xi^{\nu}_{j}}\phi(x,\xi^{\nu}_j)|+2^{j\rho_0}|\nabla_{\xi}\phi(x,\xi^{\nu}_j)|)^{-4N}dx\biggl)^{\frac 12}
	\|T_{j,k_1,\alpha}(2^{j|\beta|\rho}(\chi_{B^{\nu}_{j}}\partial^{\beta}_{\xi}\widehat{b})^{\vee})\|_2\nonumber\\
	\lesssim &\sum_{2^{j\rho}r<1,j>0}\sum^{J}_{\nu=1}2^{-\frac j2(\rho+(n-1)\rho_0)}2^{\frac{j(\rho-n)}{2}}
	\left(\|\chi_{B^{\nu}_{j}}\widehat{b}\|_2+\sum_{1\leq|\beta|\leq 2N}\|2^{j|\beta|\rho}\chi_{B^{\nu}_{j}}\partial^{\beta}_{\xi}\widehat{b}\|_2\right)\nonumber\\
	\lesssim &\sum_{2^{j\rho}r<1,j>0}2^{-\frac j2(\rho+(n-1)\rho_0)}2^{\frac{j(\rho-n)}{2}}
	\left(\sum^{J}_{\nu=1}\|\chi_{B^{\nu}_{j}}\widehat{b}\|_2+\sum^{J}_{\nu=1}\sum_{1\leq|\beta|\leq 2N}(2^{j\rho}r)^{|\beta|}|B^{\nu}_{j}|^{\frac 12}\right)\nonumber\\
	\lesssim &\sum_{2^{j\rho}r<1,j>0}2^{-\frac j2(n+(n-1)\rho_0)}\left(J^{\frac 12}\|\sum^{J}_{\nu=1}\chi_{B^{\nu}_{j}}|\widehat{b}|\|_2+J2^{j\rho}r2^{\frac j2(n-(n-1)\rho_0)}\right)\nonumber\\
	\lesssim &\sum_{2^{j\rho}r<1,j>0}2^{-\frac {jn}{2}}\|\sum^{J}_{\nu=1}\chi_{B^{\nu}_{j}}|\widehat{b}|\|_2+\sum_{2^{j\rho}r<1,j>0}2^{j\rho}r\lesssim 1.\label{gsh3.6}
\end{align}
In the last inequality, the assumption $\rho>0$ is necessary.

\subsection{Part III: $\sum\limits_{2^{j\rho}r\geq 1,j>0}\|T_{\phi,a_j}b\|_1$}

 Let $R^{\nu}_{j,r}$ be the ``exceptional set" defined in section 2.

When $0<\rho<1,\max\{\rho,1-\rho\}\leq\delta\leq 1$, from $\rho_0=\lambda_{\rho}\rho\leq \frac 12,\lambda_{\rho}\leq $, Lemma 3.3 and the Plancherel theorem, we can show that
\begin{align}
	\sum_{2^{j\rho}r\geq 1,j>0}\sum^{J}_{\nu=1}\|T^{\nu}_{j}b\|_{L^1(R^{\nu}_{j,r})}
	=&\sum_{2^{j\rho}r\geq 1,j>0}\sum^{J}_{\nu=1}\|T^{\nu}_{j}((\chi_{B^{\nu}_{j}}\widehat{b})^{\vee})\|_{L^1(R^{\nu}_{j,r})}\nonumber\\
	\leq &\sum_{2^{j\rho}r\geq 1,j>0}\sum^{J}_{\nu=1}|R^{\nu}_{j,r}|^{\frac 12}\|T^{\nu}_{j}((\chi_{B^{\nu}_{j}}\widehat{b})^{\vee})\|_2\nonumber\\
	\lesssim &\sum_{2^{j\rho}r\geq 1,j>0}\sum^{J}_{\nu=1}r^{\frac 12+\frac {(n-1)\lambda_{\rho}}{2}}2^{\frac{j(\rho-n)}{2}}\|(\chi_{B^{\nu}_{j}}\widehat{b})^{\vee}\|_2 \nonumber\\
	\lesssim &r^{\frac 12+\frac {(n-1)\lambda_{\rho}}{2}}\sum_{2^{j\rho}r\geq 1,j>0}\sum^{J}_{\nu=1}2^{\frac{j(\rho-n)}{2}}\|\chi_{B^{\nu}_{j}}\widehat{b}\|_2\nonumber\\
	\lesssim &r^{\frac 12+\frac {(n-1)\lambda_{\rho}}{2}}\sum_{2^{j\rho}r\geq 1,j>0}2^{\frac{j(\rho-n)}{2}}J^{\frac 12}\|\sum^{J}_{\nu=1}\chi_{B^{\nu}_{j}}|\widehat{b}|\|_2\nonumber\\
	\lesssim &r^{\frac 12+\frac {(n-1)\lambda_{\rho}}{2}}\sum_{2^{j\rho}r\geq 1,j>0}2^{\frac{j(\rho-n)}{2}}2^{\frac {j(n-1)\rho_0}{2}}\|b\|_2\nonumber\\
	\lesssim &r^{\frac {(n-1)(\lambda_{\rho}-1)}{2}}\sum_{2^{j\rho}r\geq 1,j>0}2^{\frac j2(\rho+(n-1)\rho_0-n)}\nonumber\\
	\lesssim &r^{\frac {(n-1)(\lambda_{\rho}-1)}{2}}\min\{1,r^{\frac{n-\rho-(n-1)\rho_0}{2\rho}}\}\nonumber\\
	\lesssim &\min\{ r^{\frac {(n-1)(\lambda_{\rho}-1)}{2}},r^{\frac{n(1-\rho)}{2\rho}} \}\lesssim 1.\label{gsh3.8}
\end{align}
Here $\rho+(n-1)\rho_0-n<0$ is necessary, which can be derived from the assumption $\rho<1$.

On the other hand, let $L$ be the operator defined in section 4.1. For any $N\in \mathbb{N}$ and $x\notin R^{\nu}_{j,r}$, by using (\ref{gsh2.5}), we can deduce that
\begin{align*}
	|T^{\nu}_{j} b(x)|=&\biggl|\int_{|y|<r}\int_{\mathbb{R}^{n}}e^{i(\nabla_{\xi}\phi(x,\xi^{\nu}_j)-y)\cdot \xi}e^{ih_j^{\nu}(x,\xi)} a_j(x,\xi)\psi^{\nu}_{j}(\xi)d\xi b(y)dy\biggl|\\
	= &\biggl|\int_{|y|<r}(1+2^{2j\rho}|\xi^{\nu}_j\cdot (\nabla_{\xi}\phi(x,\xi^{\nu}_j)-y)|^2+2^{2j\rho_0}|\nabla_{\xi}\phi(x,\xi^{\nu}_j)-y|^2)^{-N}\\
	&\int_{\mathbb{R}^{n}}L^N(e^{i(\nabla_{\xi}\phi(x,\xi^{\nu}_j)-y)\cdot \xi})e^{ih_j^{\nu}(x,\xi)} a_j(x,\xi)\psi^{\nu}_{j}(\xi)d\xi b(y)dy\biggl|\\
	\lesssim &\int_{|y|<r}(1+2^{j\rho}|\xi^{\nu}_j\cdot (\nabla_{\xi}\phi(x,\xi^{\nu}_j)-y)|+2^{j\rho_0}|\nabla_{\xi}\phi(x,\xi^{\nu}_j)-y|)^{-2N}|b(y)|dy\\
	&\biggl|\int_{\mathbb{R}^{n}}e^{i(\nabla_{\xi}\phi(x,\xi^{\nu}_j)-y)\cdot \xi}L^N(e^{ih_j^{\nu}} a_j\psi^{\nu}_{j})d\xi\biggl|\\
	\lesssim &\int_{|y|<r}(1+2^{j\rho}|\xi^{\nu}_j\cdot \nabla_{\xi}\phi(x,\xi^{\nu}_j)|+2^{j\rho_0}|\nabla_{\xi}\phi(x,\xi^{\nu}_j)|)^{-2N \lambda_{\rho}}|b(y)|dy\\
	&\biggl|\int_{\mathbb{R}^{n}}e^{i\nabla_{\xi}\phi(x,\xi^{\nu}_j)\cdot \xi}L^N(e^{ih_j^{\nu}} a_j\psi^{\nu}_{j})\left(e^{-iy\cdot \xi}\chi_{B^{\nu}_{j}}\right)d\xi\biggl|.
\end{align*}
From Lemma 2.2, we can see that $\title{a}_N(x,\xi)=L^N(e^{ih_j^{\nu}} a_j\psi^{\nu}_{j})$ satisfies the condition in Lemma 3.2 for $m=m_1$.  By using the same proof of Lemma 3.3, we obtain that
\begin{align*}
	&\int_{\mathbb{R}^{n}}\biggl|\int_{\mathbb{R}^{n}}e^{i\nabla_{\xi}\phi(x,\xi^{\nu}_j)\cdot \xi}L^N(e^{ih_j^{\nu}} a_j\psi^{\nu}_{j})\left(e^{-iy\cdot \xi}\chi_{B^{\nu}_{j}}\right)d\xi\biggl|^2dx\\
	\leq &C2^{j(\rho-n)}\int_{\mathbb{R}^{n}}|e^{-iy\cdot \xi}\chi_{B^{\nu}_{j}}|^2d\xi\\
	\leq &C2^{j(\rho-(n-1)\rho_0)},
\end{align*}
where $C$ is independent of $y,j,\nu$.

Hence, when $2\lambda_{\rho}N>n$ and $2^{j\rho}r>1$, by (\ref{gsh2.1}) we have
\begin{align}
	&\sum_{2^{j\rho}r\geq 1,j>0}\sum^{J}_{\nu=1}\|T^{\nu}_{j}b\|_{L^1((R^{\nu}_{j,r})^c)}\nonumber\\
	\lesssim &\sum_{2^{j\rho}r\geq 1,j>0}\sum^{J}_{\nu=1}\int_{|y|<r}\biggl[\int_{(R^{\nu}_{j,r})^c}(1+2^{j\rho}|\xi^{\nu}_j\cdot \nabla_{\xi}\phi(x,\xi^{\nu}_j)|+2^{j\rho_0}|\nabla_{\xi}\phi(x,\xi^{\nu}_j)|)^{-2N \lambda_{\rho}}\nonumber\\
	&|\int_{\mathbb{R}^{n}}e^{i\nabla_{\xi}\phi(x,\xi^{\nu}_j)\cdot \xi}L^N(e^{ih_j^{\nu}} a_j\psi^{\nu}_{j})\left(e^{iy\cdot \xi}\chi_{B^{\nu}_{j}}\right)d\xi|dx\biggl]|b(y)|dy\nonumber\\
	\lesssim &\sum_{2^{j\rho}r\geq 1,j>0}\sum^{J}_{\nu=1}\int_{|y|<r}\biggl[\int_{(R^{\nu}_{j,r})^c}(1+2^{j\rho}|\xi^{\nu}_j\cdot \nabla_{\xi}\phi(x,\xi^{\nu}_j)|+2^{j\rho_0}|\nabla_{\xi}\phi(x,\xi^{\nu}_j)|)^{-2n}dx
	\biggl]^{\frac 12}\nonumber\\
	&\biggl[\int_{\mathbb{R}^{n}}|\int_{\mathbb{R}^{n}}e^{i\nabla_{\xi}\phi(x,\xi^{\nu}_j)\cdot \xi}L^N(e^{ih_j^{\nu}} a_j\psi^{\nu}_{j})\left(e^{iy\cdot \xi}\chi_{B^{\nu}_{j}}\right)d\xi|^2dx\biggl]^{\frac 12}|b(y)|dy\nonumber\\
	\lesssim &\sum_{2^{j\rho}r\geq 1,j>0}\sum^{J}_{\nu=1}2^{\frac{j(\rho-(n-1)\rho_0)}{2}}\nonumber\\
	&\int_{|y|<r}\biggl[(\int_{|\xi^{\nu}_{j}\cdot z|>3r}+\int_{|z|>3r^{\lambda_{\rho}}})(1+2^{j\rho}|\xi^{\nu}_{j}\cdot z|+2^{j\rho_0}|z|)^{-2n}dx\biggl]^{\frac 12}|b(y)|dy\nonumber\\
	\lesssim &\sum_{2^{j\rho}r\geq 1,j>0}[(2^{j\rho}r)^{\frac 12-n}+(2^{j\rho}r)^{-\frac {\lambda_{\rho}n}{2}}]\lesssim 1.\label{gsh3.9}
\end{align}
Here we use the assumption $\rho,\lambda_{\rho}>0$.

Finally, from (\ref{gsh3.6}), (\ref{gsh3.8}) and (\ref{gsh3.9}), we get that
\begin{align*}
	&\|T_{\phi,a}b\|_1\\
	\lesssim &\|T_{\phi,a_0}b\|_1+\sum_{2^{j\rho}r<1,j>0}\sum^{J}_{\nu=1}\|T^{\nu}_{j}b\|_1+\sum_{2^{j\rho}r\geq 1,j>0}\sum^{J}_{\nu=1}\|T^{\nu}_{j}b\|_{L^1(R^{\nu}_{j,r})}
	+\sum_{2^{j\rho}r\geq 1,j>0}\sum^{J}_{\nu=1}\|T^{\nu}_{j}b\|_{L^1((R^{\nu}_{j,r})^c)}\\
	\lesssim &1.
\end{align*}
This finishes the proof of Theorem 1.1.

\section{Proof of Theorem 1.2}

\hspace{6mm} We use the same arguments in \cite[p. 411]{S93} to obtain Theorem 1.2. Let $p^{\prime}$ be the conjugate exponent of $p$. i.e. $\frac 1p+\frac {1}{p^{\prime}}=1$.
Note that the assumption $\max\{\rho,1-\rho\}\leq \delta<1$ implies that $0<\rho\leq \delta<1$.

For $a\in S^{m_p}_{\rho,\delta}$, we set
\begin{equation*}
	T_zf(x)=e^{(z+1-\frac 2p)^2}\int_{\mathbb{R}^{n}}e^{i\phi(x,\xi)}a(x,\xi)\left(\left(1+|\xi|^2\right)^\frac 12\right)^{\frac {n(\rho-\delta)}{2}-m_p+(m_1-\frac {n(\rho-\delta)}{2})z}\widehat{f}(\xi)\textrm{d}\xi.
\end{equation*}
It is easy to see that $\frac {n(\rho-\delta)}{2}-m_p+(m_1-\frac {n(\rho-\delta)}{2})(\frac 2p-1)=0$ and $T^*_z$ is holomorphic in $z$ (see \cite[p. 175]{S93}). Thus we have
$$T_{\frac 2p-1} f(x)=T_{1-\frac {2}{p^{\prime}}} f(x)=\int_{\mathbb{R}^{n}}e^{i\phi(x,\xi)}a(x,\xi)\widehat{f}(\xi)d\xi=T_{\phi,a} f(x).$$
For any $t\in \mathbb{R}$, we get that
\begin{align*}
	&|T_{it}f(x)|=e^{(\frac 2p-1)^2}\left|\int_{\mathbb{R}^{n}}e^{i\phi(x,\xi)}e^{-t^2}\left(1+|\xi|^2\right)^{\frac {it(2m_1-n(\rho-\delta))}{4}}a(x,\xi)\left(\left(1+|\xi|^2\right)^\frac 12\right)^{\frac {n(\rho-\delta)}{2}-m_p}\widehat{f}(\xi)d\xi\right|,\\
	&|T_{1+it} f(x)|=e^{(2-\frac 2p)^2}
	\left|\int_{\mathbb{R}^{n}}e^{i\phi(x,\xi)}e^{-t^2}\left(1+|\xi|^2\right)^{\frac {it(2m_1-n(\rho-\delta))}{4}}a(x,\xi)\left(\left(1+|\xi|^2\right)^\frac 12\right)^{m_1-m_p}\widehat{f}(\xi)d\xi\right|.
\end{align*}
Set
$$b(x,\xi)=e^{-t^2}\left(1+|\xi|^2\right)^{\frac {it(2m_1-n(\rho-\delta))}{4}}a(x,\xi)\left(\left(1+|\xi|^2\right)^\frac 12\right)^{\frac {n(\rho-\delta)}{2}-m_p}.$$
Some simple computations yield that
$b\in S^{\frac {n(\rho-\delta)}{2}}_{\rho,\delta}$ and the semi-norms do not depend on $t$ (see \cite[p. 411]{S93}).
By \cite[Theorem 2.7]{FS14} we have
$$\|T^*_{it} f\|_2\lesssim \|f\|_2.$$
Similarly, by Theorem 1.1 and $H^1-BMO$ duality argument, we can obtain that
$$\|T^*_{1+it} f\|_{BMO}\lesssim  \|f\|_{\infty}.$$
Using the Fefferman-Stein interpolation \cite[p. 175]{S93}, for $1<p<2$ we get that
$$\|T^*_{\phi,a} f\|_{p^{\prime}}=\|T^*_{\frac 2p-1}f\|_{p^{\prime}}=\|T^*_{1-\frac {2}{p^{\prime}}}f\|_{p^{\prime}}\lesssim\|f\|_{p^{\prime}}.$$
Therefore, by duality argument, we prove that
$$\|T_{\phi,a} f\|_{p}\lesssim\|f\|_{p}$$
as desired.\\


\begin{thebibliography}{99}
	\bibitem{AH90} \'{A}lvarez J. and Hounie J.,
	Estimates for the kernel and continuity properties of pseudo-differential operators.
	\textbf{Ark. Mat.} 28 (1990),  1-22.
	
	\bibitem{Af78} Asada K. and Fujiwara D.,
	On some oscillatory integral transformations in $L^{2}(\mathbb{R}^n)$.
	\textbf{Japan. J. Math. (N.S.).} 4 (1978), 299-361.
	
	\bibitem{B97} Boulkhemair A.,
	Estimations $L^{2}$ pr\'{e}cis\'{e}es pour des int\'{e}grales oscillantes, (French) [$L^{2}$-estimates for oscillating integrals].
	\textbf{Comm. Partial Differential Equations}  22 (1997), 165-184.
	
	\bibitem{CV71} Calder\'{o}n A. P. and Vaillancourt R.,  On the boundedness of pseudo-differential operators.
	\textbf{ J. Math. Soc. Japan} 23 (1971), 374-378.
	
	\bibitem{CV72} Calder\'{o}n A. P. and Vaillancourt R., A class of bounded pseudo-differential operators. \textbf{Proc. Nat. Acad. Sci. U.S.A.} 69 (1972), 1185-1187.
	
	\bibitem{CIS21} Castro A., Israelsson A. and Staubach  W.,
	Regularity of Fourier integral operators with amplitudes in general H\"{o}rmander classes.
	\textbf{Anal. Math. Phys.} 11 (2021), no.3, Paper No. 121, 54 pp.
	
	\bibitem{CNR09} Cordero E., Nicola F. and Rodino L.,
	Boundedness of Fourier integral operators on $\mathcal{F}L^p$ spaces.
	\textbf{Trans. Amer. Math. Soc.} 361 (2009), 6049-6071.
	
	\bibitem{CNR10} Cordero E.,  Nicola F. and Rodino L.,
	On the global boundedness of Fourier integral operators.
	\textbf{Ann. Global Anal. Geom.} 38 (2010), 373-398.
	
	\bibitem{CR10} Coriasco S. and Ruzhansky M.,
	On the boundedness of Fourier integral operators on $L^{p}(\mathbb{R}^n)$.
	\textbf{C. R. Math. Acad. Sci. Paris.} 348 (2010), 847-851.
	
	\bibitem{CR14} Coriasco S. and Ruzhansky M.,
	Global $L^p$-continuity of Fourier integral operators.
	\textbf{Trans. Amer. Math. Soc.} 366 (2014), 2575-2596.
	
	\bibitem{DGZ23} Dai J., Guo J. and Zhu X.,
	$L^p$ boundedness of Fourier integral operators with rough symbols.
	\textbf{J. Math. Anal. Appl.} 517 (2023), 126654 (14pp).
	
	\bibitem{FS14} Dos Santos Ferreira D. and Staubach W.,
	Global and local regularity of Fourier integral operators on weighted and unweighted spaces.
	\textbf{Mem. Amer. Math. Soc. } 229 (2014), xiv+65pp.
	
	\bibitem{E70} \`{E}skin G.I.,
	Degenerate elliptic pseudodifferential equations of principal type (Russian).
	\textbf{Mat. Sb. (N.S.).} 82 (1970), no. 124, 585-628.
	
	\bibitem{F73} Fefferman C.,
	$L^p$ bounds for pseudo-differential operators.
	\textbf{Israel J. Math.} 14 (1973), 413-417.
	
	\bibitem{F78} Fujiwara D.,
	A global version of Eskin's theorem.
	\textbf{J. Fac. Sci. Univ. Tokyo Sect. IA Math.} 24 (1977), 327-339.
	
	\bibitem{G79} Goldberg D., A local version of real Hardy spaces. \textbf{Duke Math. J.} 46 (1979), no. 1, 27-42.
	
	\bibitem{GZ22} Guo  J. and Zhu X.,
	Some notes on endpoint estimates for pseudo-differential operators.
	\textbf{Mediterr. J. Math.} 19 (2022), no. 6, Paper No. 260, 14 pp.
	
	\bibitem{HPR20} Hassell A.,  Portal P. and Rozendaal J.,
	Off-singularity bounds and Hardy spaces for Fourier integral operators.
	\textbf{Trans. Amer. Math. Soc.} 373 (2020),  5773-5832.
	
	\bibitem{H71CPAM} H{\"o}rmander L.,
	On the $L^{2}$ continuity of pseudo-differential operators.
	\textbf{Comm. Pure Appl. Math.} 24 (1971), 529-535.
	
	\bibitem{H71AM} H{\"o}rmander L.,  Fourier integral operators. I.
	\textbf{Acta Math.} 127 (1971), 79-183.
	
	\bibitem{H86} Hounie J.,
	On the $L^2$-continuity of pseudo-differential operators.
	\textbf{Comm. Partial Differential equations} 11 (1986), 765-778.
	
	\bibitem{IRS21} Israelsson A., Rodr\'{i}guez-L\'{o}pez S. and Staubach W.,
	Local and global estimates for hyperbolic equations in Besov-Lipschitz and Triebel-Lizorkin spaces.
	\textbf{Anal. PDE.} 14 (2021), no.1, 1-44.
	
	\bibitem{IMS23} Israelsson A., Mattsson T. and Staubach W.,
	Boundedness of Fourier integral operators on classical function spaces.
	\textbf{J. Funct. Anal.} 285 (2023), no. 5, Paper No. 110018, 64 pp.
	
	\bibitem{K76} Kumano-go H.,
	A calculus of Fourier integral operators on $\mathbb{R}^n$ and the fundamental solution for an operator of hyperbolic type.
	\textbf{Comm. Partial Differential Equations} 1 (1976), 1-44.
	
	\bibitem{R76} Rodino L.,
	On the boundedness of pseudo differential operators in the class $L^{m}_{\rho,1}$.
	\textbf{Proc. Amer. Math. Soc. } 58 (1976), 211-215.
	
	\bibitem{RS06} Ruzhansky M. and Sugimoto M.,
	Global $L^{2}$-boundedness theorems for a class of Fourier integral operators.
	\textbf{Comm. Partial Differential Equations} 31 (2006), 547-569.
	
	\bibitem{RS19} Ruzhansky M. and Sugimoto M.,
	A local-to-global boundedness argument and Fourier integral operators.
	\textbf{J. Math. Anal. Appl.} 473 (2019), no. 2, 892-904.
	
	\bibitem{SSS91} Seeger A., Sogge C. D. and Stein E. M., Regularity properties of Fourier integral operators.
	\textbf{Ann. of Math.}  134 (1991), 231-251.
	
	\bibitem{SZ23} Shen S. and Zhu X.,
	Fourier integral operators on $L^p(\mathbb{R}^n)$ when $2<p\leq\infty$.
	\textbf{Anal. Math. Phys.} 13 (2023), no. 4, Paper No. 62, 20 pp.
	
	\bibitem{S93} Stein E. M.,
	Harmonic analysis: real-variable methods, orthogonality, and oscillatory integrals. With the assistance of Timothy S. Murphy.
	\textbf{Princeton Mathematical Series, 43. Monographs in Harmonic Analysis, III.} Princeton University Press, Princeton, NJ, 1993.
	
	\bibitem{T04} Tao T., The weak-type $(1,1)$ of Fourier integral operators of order $-(n-1)/2$.
	\textbf{J. Aust. Math. Soc.} 76 (2004), 1-21.
	
	\bibitem{YZZ23} Ye X., Zhang C. and Zhu X., Fourier integral operators on Hardy spaces with amplitudes in forbidden H\"{o}rmander class.
	\textbf{Preprint.} https://arxiv.org/abs/2406.03076
	
	
	
\end{thebibliography}
\end{document}